\documentclass{amsart}
\usepackage{latexsym,amsfonts}
\usepackage{amssymb}

\def\GL{\mathop{\rm GL}\nolimits}

\def\SL{\mathop{\rm SL}\nolimits}
\def\GL{\mathop{\rm GL}\nolimits}            
\def\PSL{\mathop{\rm PSL}\nolimits}          
\def\PSp{\mathop{\rm PSp}\nolimits}           
           
\def\SU{\mathop{\rm SU}\nolimits}          
\def\PSU{\mathop{\rm PSU}\nolimits}

\def\dim{\ \mathop{\rm dim}\nolimits}

\def\Alt{\mathop{\rm Alt}\nolimits}

\def\id{\mathrm{id}}

\newtheorem{theorem}{Theorem}[section] 
\newtheorem{lemma}[theorem]{Lemma}     

\newtheorem{proposition}[theorem]{Proposition}

\newcommand{\F}{\mathbb{F}}    

\begin{document}
\title{The $(2,3)$-generation of the
special unitary groups of dimension $6$}
\author{M.A. Pellegrini}
\address{Dipartimento di Matematica e Fisica, Universit\`a Cattolica del Sacro Cuore, Via Musei 41,
I-25121 Brescia, Italy}
\email{marcoantonio.pellegrini@unicatt.it}

\author{M. Prandelli}
\email{mariateresa.prandelli@istruzione.it}

\author{M.C. Tamburini Bellani}
\address{Dipartimento di Matematica e Fisica, Universit\`a Cattolica del Sacro Cuore, Via Musei 41,
I-25121 Brescia, Italy}
\email{mariaclara.tamburini@unicatt.it}

\keywords{Generation; Unitary groups}
\subjclass[2010]{20G40, 20F05}

\begin{abstract}
In this paper we give explicit $(2,3)$-generators of the unitary groups $\SU_6(q^ 2)$,
for all $q$.
They fit into a uniform sequence of likely $(2,3)$-generators for all $n\ge 6$.
\end{abstract}
\maketitle

\section{Introduction}

A $(2,3)$-generated group is a group that can be generated by two elements of respective orders $2$ and $3$.
Apart from the infinite families $\PSp_4(2^ a)$ and $\PSp_4(3^ a)$, the other finite classical simple groups are
$(2,3)$-generated, up to a finite number of exceptions. This fact, proved in \cite {LS}
with probabilistic methods, was suggested by the constructive positive results obtained, with permutational methods,
for classical groups of sufficiently large rank (e.g. see \cite{T}, \cite{DV96}, \cite{TWG}, \cite{TW}). 
A natural question is to detect all the exceptions (Problem 18.49 in \cite{MK}). For
$n\leq 5$ the complete list of exceptions is:
$\PSL_2(9)$,  $\PSL_3(4)$, $\PSL_4(2)$, $\PSU_3(q^2)$ for $q=2,3,5$,  
$\PSU_4(q^2)$ for $q=2,3$, and $\PSU_5(4)$. Moreover P$\Omega^+_8(2)$ and P$\Omega^+_8(3)$ 
are not $(2,3)$-generated (M. Vsemirnov, 2012). According to what conjectured in  \cite{Vsur}, 
this should be the complete list.

The most obvious step towards the confirmation of this claim is to consider the cases 
$n=6,7$. The following groups are known to be $(2,3)$-generated:
\begin{itemize}
\item $\PSL_6(q)$ for all $q$ (see \cite{DV96} for $p\ne 2$, $q\ne 9$ and \cite{T6} 
 for all $q$);
\item $\PSp_6(q)$  for all $q$ (see \cite{TV6} for $q$ odd and \cite{Pel} for $q$ even);
\item  $\PSL_7(q)$ for all $q$ (see \cite{DV96} for $p\ne 2$, $q\ne 9$ and \cite{T7}  for all $q$);
\item $\PSU_7(q^2)$ and $\Omega_7(q)$  for all $q$ (see \cite{Pel});
\end{itemize}

From Proposition \ref{U62} and Theorem \ref{main6} of this paper follows Theorem \ref{main},
which closes the only case left open for $n\le 7$, namely the $6$-dimensional unitary groups.

\begin{theorem}\label{main}
The groups $\SU_6(q^2)$ and $\PSU_6(q^2)$ are $(2,3)$-generated for all $q$. 
\end{theorem}

Noting that $\SU_n(4)$ is not $(2,3)$-generated for $3\le n <6$, the value $n=6$  seems a good lower bound
for the existence of uniform $(2,3)$-generators for the unitary groups. 
Indeed our $(2,3)$-generators $x=x_6$, $y=y_6$ of $\SU_6(q^2)$, described  in \eqref{generatorsU6} for $q> 2$, 
in \eqref{genF4} for $q=2$, fit into a sequence of
$(2,3)$-elements $x_n, y_n\in \SU_n(q^2)$, $n\ge 6$ (see the last Section).
These elements, for $q>2$,  depend on a parameter $a\in \F_{q^2}\setminus \F_q$. 
Computer evidences, combined with the permutational methods mentioned above,
strongly suggest that $\left\langle x_n, y_n\right\rangle = \SU_n(q^2)$,
whenever $\langle x_6,y_6\rangle=\SU_6(q^2)$. For $q>2$ a sufficient condition is that $a$  
satisfies the assumptions of Theorem \ref{main6}.

Throughout this paper,  $\F$ is an algebraically closed field of
characteristic $p\ge 0$. The set $\{e_1,e_2,\ldots,e_n\}$ denotes the canonical basis of
$\F^n$. When $p>0$ we have that $\F_{q}\leq \F$ for any fixed power $q$ of the prime
$p$. Further, we denote by $\sigma$ the automorphism of $\SL_n(q^2)$ defined by
$
\left(\alpha_{i,j}\right)\mapsto \left(\alpha_{i,j}^q\right).
$
Finally, $\omega$ is an element of $\F$ of order $3$ if $p\ne 3$, $\omega =1$ if $p=3$.

\section{The groups $\SU_6(q^ 2)$} 

In order to prove that the groups $\SU_6(q^2)$ are $(2,3)$-generated, we construct our generators. 
Let $a\in \F_{q^2}\setminus \F_q$ (in particular $a\neq 0$) be such that $a^{q+1}\ne 4$.
Also set:
$$\left.
\begin{array}{ll}
b  = 2a-a^{2q},\quad & c= a^{q+1}-4, \\
d =  2a^3+2a^{3q}-12a^{q+1}+16,\quad & e  =  -a^{4q}+6a^{2q+1}-a^{q+3}-8a^q.
\end{array}\right.$$
Consider now the subgroup $H=\langle x, y\rangle$ of $\SL_6(q^ 2)$, where 

\begin{equation}\label{generatorsU6}
x=x_6=\begin{pmatrix}
 1 &0&0&0&0&0\\
0& 1&0&0&0&0\\
0&0&0&1&0&0\\
0&0&1&0&0&0\\
0&0&0&0&-1&a\\
0&0&0&0&0&1
\end{pmatrix},\quad y=y_6=\begin{pmatrix}
0&0&1&0&0&0\\
1&0&0&0&0&0\\
0&1&0&0&0&0\\
0&0&0&\frac{b}{c}&\frac{e}{c^2}&\frac{d}{c^2}\\
0&0&0&1&-\frac{b}{c}&-\frac{b^q}{c}\\
0&0&0&0&1&0
\end{pmatrix} . 
\end{equation}

Notice that 
\begin{equation}\label{cond}
(a^ 3-8)=(a-2)(a^ 2+2a+4)\neq 0.
\end{equation}
Namely, we are excluding  $a=2$, since $a \in \F_{q^ 2}\setminus \F_q$. So $a^ 3=8$
implies $a=2\omega^j$ and $a^q=2\omega^{2j}$ ($j=1,2$).
However, in this case $c=a^ {q+1}-4=0$, against our initial assumptions.
\medskip

The similarity invariants of $x,y$ are, respectively, 
$$t-1,\; t-1,\; t^2-1, \;t^2-1 \quad \textrm{ and } \quad 
t^3-1,\; t^3-1.$$
The characteristic polynomial of $z=xy$ is
\begin{equation}\label{char6}
\chi_z(t)= t^ 6- \frac{b+ac}{c} t^ 5 - \frac{b^ q}{c} t^ 4  - \frac{b}{c}   t^ 2  -
\frac{b^q+ ca^ q }{c} t +1 .
\end{equation}

\begin{lemma}\label{unitary}
If $H=\left\langle x,y\right\rangle$ is absolutely irreducible, then  the characteristic
polynomial $\chi_z(t)$ of $z$ coincides with its  minimal polynomial and $H\leq \SU_6(q^
2)$.
\end{lemma}

\begin{proof}
We have $\dim C(x)=16+4=20$, $\dim C(y)=4+4+4=12$.
By Scott's formula, when $H$ is absolutely irreducible we get $\dim C(z)\le
(6^2+2)-20-12=6$. 
 From the Frobenius formula giving the dimension of centralizers of the rational canonical
forms,
we see that, for a fixed $n$, the minimal dimension is $n$, and it is attained precisely
by
the canonical forms having just one similarity invariant.
It follows that $z$ has a 
unique similarity invariant, whence our first claim.
In particular, this means that the triple $(x,y,z)$ is rigid and by \cite[Theorem
3.1]{PTV} we obtain $H\leq \SU_6(q^2)$.
\end{proof}

It is useful the Gram matrix of the hermitian form fixed by $H$, namely:
$$J=\frac{1}{c} \left(
\begin{array}{c|c}
   c^ 2\cdot I_4 & 0 \\ \hline
   0 & \begin{matrix}
   d & e\\
 e^q &d 
    \end{matrix}
    \end{array}
\right),
$$
where $J^T=J^\sigma$, $x^TJx^\sigma =J$ and $y^TJy^\sigma =J$.
Denoting
\begin{equation}\label{J}
\gamma_j  =  a+\omega^{-j}a^q+2\omega^j \; (j=0,1,2), \quad  \gamma  = a^3+a^{3q}-6a^{q+1}+8=
\gamma_0\gamma_1\gamma_2,
\end{equation}
we have $\det(J)=-c^3\gamma^2$.
\medskip

In the following proposition we  make use of these subspaces of $\F^ 6$:
$$A=\langle e_3-e_4, e_5\rangle,\quad  B=\langle e_3-e_4, 2e_5-ae_6\rangle,$$
$$V_x=\langle e_1, e_2, e_3+e_4, ae_5+2e_6 \rangle ,\qquad V_{x^ T}=\langle e_1, e_2,
e_3+e_4, e_6\rangle.$$
Notice that $V_x$ and $V_{x^ T}$ are, respectively, the eigenspace of $x$ and $x^ T$
associated to the eigenvalue $1$.
For $p\neq 2$, $A$ and $B$  are, respectively, the eigenspace of $x$ and $x^ T$ associated
to the eigenvalue $-1$.
For $p=2$, $A\leq V_x$ and $B\leq V_{x^ T}$.

\begin{proposition}\label{irreducibility}
Let $a\in \F_{q^2}\setminus \F_q$ such that $a^{q+1}\ne 4$. The group
$H=\left\langle x,y\right\rangle$ is absolutely irreducible if, and only if, the
following conditions hold:
\begin{itemize}
 \item[(i)] $\det(J)=-c^ 3\gamma^ 2\ne 0$, i.e., $a^3 - 6a^{q+1}+ a^{3q} + 8\neq 0$;
\item[(ii)] $a^{2q+2}-5a^{q+1}+8\ne 0$.
\end{itemize}
In particular for $q=2$ the group $H$ is reducible over $\F$.
\end{proposition}

\begin{proof}
If $\gamma=0$ then the kernel of $J$ is fixed by $H$.
If $\gamma \ne 0$ and $a^{2q+2}-5a^{q+1}+8=0$, then $a^{2q}-a\ne 0$.
Indeed $a^{2q}-a= 0$ would give  $a^q=a^2$, whence  $a^{2q+2}-5a^{q+1}+8=
a^3-6a^{q+1}+a^{3q}+8=\gamma$. In this case the subspace $W$
generated by $\{v,yv,y^2v\}$ with $v=\left(0,0,1,-1, \frac{c}{a^{2q}-a},0\right)^T$
is $H$-invariant. So conditions (i) and (ii) are necessary.
Now we show that they are also sufficient.

Assume that (i) and (ii) hold and let $W$ be a proper $H$-invariant space.\\
\textbf{Case (a)}  $\dim(W)\le 3$.\\
\textbf{(a.1)} There exists $0\ne v\in A\cap W$. Set $v=(0,0,x_1,-x_1,x_2,0)^T$.\\
We show
that $v, yv, y^2v$ are linearly independent. 
We consider the matrix $A(i,j,k)$ with rows $v[s], yv[s], y^2v[s]$ for
$s\in\left\{i,j,k\right\}$
and call $\delta (i,j,k)$ its determinant.
If $x_1=0$, we may assume $x_2=1$. Then $\delta(4,5,6)=\frac{\gamma(a^{3q}-8)}{c^ 3}$ which
is non-zero.
In this case $xyv\in W$ iff $a^q\gamma=0$, which is an absurd.
If $x_1\ne 0$, then we may set $x_1=1$. In this case $\delta(1,2,3)=1$. 
Thus $\left\langle v, yv, y^2v\right\rangle$ has dimension $3$ and coincides with $W$. 
We have that $xyv\in W$ iff $w_1=xyv-yv=\lambda v$, where $\lambda$ is the coordinate of 
position $3$ in $w_1$.  Similarly $xy^2v\in W$ iff $w_2=xy^2v-y^2v=\mu v$  where $\mu$ is
the coordinate of  position $3$ in $w_2$. 
Let $p_1$, $p_2$ denote the coordinates of position $5$ 
of $w_1-\lambda v$ and $w_2-\mu v$ respectively. Since
all the remaining coordinates are 0, we have that
$W$ is $H$-invariant only if $p_1=p_2=0$.
In particular $2p_1+a^qp_2=0$ gives $\left(a^{2q}-a\right)x_2-c=0$.
 From the assumption $c\ne 0$ it follows $a^{2q}-a\ne 0$, whence
$x_2=\frac{c}{a^{2q}-a}$. After this
substitution direct calculation shows that $p_1=p_2=0$ iff $a^{2q+2}-5a^{q+1}+8=0$.\\
\noindent \textbf{(a.2)}  $A\cap W=0$.\\
Clearly, there exists $0\ne v\in V_{1}^x\cap W$. 
Set $v=(x_1,x_2,x_3,x_3,ax_4,2x_4)^T$ and 
$u_1=yv-xyv$, $u_2=y^2v-xy^2v$. 
Observe that  $u_1,u_2\in A\cap W$, hence
$u_1=u_2=0$. The coordinates of position $5$ of $u_1$, $u_2$
are, respectively, 
$ 2x_3+ (2a^ q-a^2)x_4$ and $- ax_3 -cx_ 4$.
Considering the coordinate of position $5$ of $au_1+2u_2$
we get $(a^3-8)x_4=0$, whence $x_4=0$. After substituting of this value,
the coordinate $5$ of $u_2$ becomes $-ax_3$, whence $x_3=0$. 
Thus we obtain $u_1=(0,0,x_2,-x_2,0,0)^T$,
$u_2=(0,0,x_1,-x_1,0,0)^T$. It follows $x_1=x_2=0$, whence the contradiction 
$v=0$.\\
\textbf{Case (b)}  $\dim(W)\ge 4$.\\
In this case there exists an $H^T$-invariant space $U$, with $\dim(U)\le 2$. 
Let $0\ne v\in B \cap U$ and set $v=(0,0,x_1,-x_1,2x_2,-ax_2)^T$. 
We have that  $v, y^T v, (y^T)^2v$ must be linearly dependent. This easily
implies  $x_1=0$. After this substitution, $(a^3-8)x_2^3=0$, whence the contradiction
$x_2=0$. It follows that $B\cap U=0$. 
So there exists  $0\ne v\in V_{1}^{x^T}\cap U$. Set $v=\left(x_1,x_2,x_3,x_3,0,x_4\right)$
and $u_1=y^ Tv-x^ Ty^ Tv$, $u_2=(y^T)^2 v-x^T(y^T)^ 2v$. 
As above $u_1,u_2\in B\cap U$, hence
$u_1=u_2=0$. Coordinates of position $4$ give:
$-cx_1+(2a-a^{2q})x_3=0$ and 
$(2a^q-a^2)x_3-cx_2+cx_4=0$.
If $x_3=0$ then, considering the other coordinates of $u_1$
we get $x_1=x_4=0$. In this case, $u_2=x_2(e_3-e_4)$ and so $x_2=0$ and $v=0$. 
So assume $x_3=c$. We get $x_1= 2a-a^{2q}$ and $x_2= 2a^ q-a^ 2+x_4$.
Coordinate $6$ of $u_1$ gives
$x_4=\frac{a^{q+3}+a^{4q}-6a^{2q+1}+8a^ {q}}{c}$.
Imposing that coordinate $6$ of $u_2$ is $0$, we get 
$a(a^{3q}-8)\gamma=0$, a contradiction.
\end{proof}

\begin{lemma}\label{powers} 
If $H$ is absolutely irreducible, then  $z^k$ is not scalar for   $k\le 9$.
\end{lemma}

\begin{proof}
For $k\leq 7$ our claim follows from the following observations:
$$
\left. \begin{array}{lll}
ze_1=e_2, & z^4e_1= z^5e_3=\sum_{j\ne 6}\lambda_je_j -e_6,\\
z^2e_1= z^3e_3=e_4,  &  z^6e_1=z^7e_3=\sum_{j\ne 2}\mu_je_j
+\frac{a^q\gamma}{c^2}e_2.
\end{array}
\right.
$$
Now, assume $z^8$ is scalar. From $z^ 8e_1=\frac{\gamma \alpha}{c^ 4}e_2+\frac{\gamma
\beta}{c^ 3}e_4+\sum_{j\neq 2,4}\lambda_j e_j$, we get $\alpha=\beta=0$,
where 
\begin{eqnarray*}
\alpha & = & a^{3q+4}-2 a^{4q+2}-7 a^ {2q+3}+a^{5q}+8a^ {3q+1}+ 20a^{q+2}-16a^{2q}-16a,\\ 
\beta & = & a^{2q+2}-a^{3q}-4a^ {q+1}+8. 
\end{eqnarray*}
Treating $\alpha$ and $\beta$ as polynomials in $a$ and $a^q$, their resultant
with respect to $a^q$ is  $64(a^ 3-8)^ 3$, 
whence $p=2$.  In this case $\beta=a^{2q}(a^q+a^ 2)$. Taking $a^q=a^ 2$, we
obtain  $\lambda_1=a^ 2\neq 0$.
To prove that $z^9$ is not scalar, simply note that $z^ 9e_3=z^8e_1$ and reapply the
previous computations.
\end{proof}

We need to exclude that $H$ is contained in any maximal subgroup of $\SU_6(q^ 2)$. 
We refer to  \cite[Tables 8.26 and 8.27, pages 390-391]{Ho} for the list of these subgroups.

\begin{lemma}\label{monomial}
If $H$ is absolutely irreducible, then $H$ is not monomial. 
\end{lemma}

\begin{proof} 
Let $\mathcal{B}=\{v_1,\dots \}$  be a basis of $\F^6$ on which
$H$ acts monomially. 
Since $H$ is irreducible, this action must be transitive.
By this condition and the canonical
form of  $x$ and $y$, we note that the permutation induced by $x$ cannot be  the
product of three $2$-cycles, and we may set: 
$$yv_1=v_2,\ yv_2=v_3,\ xv_3=v_4,\  yv_4=v_5,\ yv_5=v_6.$$
Thus $\mathcal{B}=\left\{v_1, v_2, v_3,v_4, v_5, v_6\right\}$. Clearly
$yv_3=v_1$, $xv_4=v_3$  and  $yv_6=v_4$.  Considering the canonical form of $x$, only the
following cases have to be analyzed. \\
\textbf{Case (a)} The permutation induced by $x$ is a $2$-cycle ($p$ odd). Thus:
$$xv_i=\delta_i v_i,\ (i=1,2,5,6),\ \delta_i=\pm 1.$$
In this case $(xy)^6$, is scalar,  a contradiction with Lemma \ref{powers}.\\
\textbf{Case (b)} The permutation induced by $x$ is the product of two  $2$-cycles
and the $1$-dimensional spaces fixed by $x$ are in different orbits of $y$. 
We may suppose either
$xv_1=\lambda v_5, \ xv_2=v_2,\ xv_6=v_6$ or $xv_1=\lambda v_6,\ xv_2=v_2, \ xv _5= v_5$.

In the first case, $\chi_z(t)=t^ 6-(\lambda+\lambda^{-1})t^3+1 $.
Comparing this polynomial with \eqref{char6} we obtain $\lambda^ 2=-1$ and so $(xy)^
6=-I$, against Lemma \ref{powers}.
In the second case, $\chi_z(t)=t^ 6-\lambda t^ 4- \lambda^ {-1} t^ 2+1$. 
Comparison with \eqref{char6} gives 
$c^ 2-b^{q+1}=2\gamma=0$ and so $p=2$.
However in this case, we must have $b+ac=0$  that gives
$a^ q=a^ 2$, whence $a^ 3=1$ and $\det(J)=0$. \\
\textbf{Case (c)}  The permutation induced by $x$ is the product of two  $2$-cycles
and the $1$-dimensional spaces fixed by $x$ are in the same orbit of $y$. We may suppose
$xv_1= v_1,\ xv_2= v_2,\ xv_5=\lambda v_6,\ xv_6=\lambda^{-1} v_5$.
Consideration of the characteristic polynomial
of $z$ gives the conditions: $b=0$ (i.e. $a^{2q}=2a$),  $\lambda=a^q$
and $\lambda^{-1}=a$. It follows
$a^{q+1}=1$ and  $2a^3=1$. 
By the assumption $a\not\in \F_q$, we have $p\ne 2,3$. 
However $1=(2a^3)(2a^{3q})$ gives the absurd
$1=4$.
\end{proof}

\begin{lemma}\label{minimal field} If $\F_p[a^3]=\F_{q^2}$ then the projective image
$\overline H$
of $H$ is not conjugate to any subgroup of
$\PSL_6(q_0)$ for any $q_0<q^2$. 
\end{lemma}

\begin{proof}
Assume that $\overline H$ is conjugate to a subgroup of $\PSL_6(q_0)$, i.e., that 
$$g^{-1}Hg\le \SL_6\left(\F_{q_0}\right)\left\langle -\omega I\right\rangle, $$
for some $g$. In particular $g^{-1}zg \omega^{-j}\in \SL_6\left(\F_{q_0}\right)$ for some
$j\in \left\{0,\pm 1\right\}$. Set $z_0=g^{-1}zg \omega^{-j}$ and
$\chi_{z_0}(t)=\sum_{k=0}^6f_kt^k$.
Clearly all $f_k$'s lie in $\F_{q_0}$. 
Then 
$$\chi_z(t)=\chi_{g^{-1}zg }(t)=\chi_{z_0\omega^j}(t)=\sum_{k=0}^6
f_k\omega^{j(6-k)}t^k.$$
Comparing the coefficients of the terms of degrees $5$ and $2$ (see \eqref{char6}), we
get 
$$-\frac{b+ac}{c}= \omega^jf_5,\quad -\frac{b}{c}=\omega^{4j}f_2= \omega^jf_2.$$
It follows $\omega^j\left(f_2-f_5\right)=a$, whence $a^3\in \F_{q_0}$.
\end{proof}

\begin{lemma} \label{tensor} 
If $H$ is absolutely irreducible, then it
is not conjugate to any maximal subgroup $M$ with
projective image $\overline M= \PSU_2(q^2)\times \PSU_3(q^2)$. 
\end{lemma}

\begin{proof}
The subgroup $\overline M$ is contained in the projective image of $\GL_2(\F)\otimes
\GL_3(\F)$. 
Assume that our claim is wrong and write $x=x_2\otimes x_3$ with 
$x_2\in \GL_2(\F)$ and $x_3\in \GL_3(\F)$.
 From  $x_2^2\otimes x_3^2=x^2=I$, we get $x_2^2=\rho I_2$ and
$x_3^2=\rho^{-1} I_3$. By the irreducibility of $H$ neither $x_2$ nor $x_3$ can be scalar.
So the rational  canonical forms of $x_2$ and $x_3$ are respectively
$$\begin{pmatrix}
0&\rho\\
1&0
\end{pmatrix},
\quad \begin{pmatrix}
\pm \sqrt {\rho^{-1}}&0&0\\
0&0&\rho^{-1}\\
0&1&0
\end{pmatrix}.$$
It follows that the rational canonical form of $x_2\otimes x_3$ consists of $3$ blocks
$\left(\begin{smallmatrix}
0&1\\
1&0
\end{smallmatrix}\right)$, in contrast with the rational canonical form of $x$.
\end{proof}

Let $M$ be a maximal subgroup  in class $\mathcal{S}$, whose projective image 
$\overline M$ is isomorphic to $\PSU_3(q^2)$, $q$ odd. Denote by $S^2(V)$ the symmetric square
of $V=\F^3$. For $\tau\in \left\{\id,\ \sigma\right\}$
consider the representation $\varphi_\tau$,  of degree 6, defined by 
$$\varphi_\tau (g):= (g^\tau \otimes g^\tau)_{|S^2(V)},\quad \forall\ g\in \SU_3(q^2).$$
Since the groups in class $\mathcal{S}$ are absolutely
irreducible, by  \cite[5.4.11]{KL} we have that, up to conjugation,
$M= \pm\thinspace  \varphi_\tau(\SU_3(q^2))$,
for some $\tau\in \left\{\id,\ \sigma\right\}$.

\begin{lemma}\label{U3} 
Assume $q$ odd and $H$  absolutely irreducible. Then $H$
is not conjugate to any maximal subgroup $M$ whose projective image
is isomorphic to $\PSU_3(q^2)$. 
\end{lemma}

\begin{proof}
Assume the contrary. Then there exist $\xi, \eta\in \SU_3(q^2)$ such that
$X:=\pm \thinspace \varphi_\tau(\xi)=x^h$, $Y:=\varphi_\tau(\eta)=y^h$,
for some $h$.  
Thus $\xi^2=\pm I$ and $\eta^3=I$. Since the space of fixed points of $x$ has dimension
$4$,
we need $X:=\varphi_\tau(\xi)$.
In particular $\xi$ has a $2$-dimensional eigenspace $U$. 
By the irreducibility we may assume
$U\cap \eta U= \left\langle e_2\right\rangle$. Setting $e_1=\eta^{-1}e_2$ we have
$e_1\in U$
and $\eta e_2\not\in U$. Thus, up to conjugation:
$$\xi=\begin{pmatrix}
-1,0,s\\
0,-1, r\\
0,0,1
\end{pmatrix},\ 
\eta =\begin{pmatrix}
0,0,1\\
1,0, 0\\
0,1,0
\end{pmatrix}.$$
Finally, for $\xi\eta$ to be unitary, we need $r=s^q$
(see \cite{PT}). With respect to the basis 
$$e_1\otimes e_1, e_2\otimes e_2, e_3\otimes e_3, \frac{e_1\otimes e_2+e_2\otimes e_1}{2},
\frac{e_1\otimes e_3+e_3\otimes e_1}{2},
\frac{e_2\otimes e_3+e_3\otimes e_2}{2},$$
of $S^2(V)$, we get (e.g. assuming $\tau=\id$):

$$ X=\begin{pmatrix}
1 & 0 & s^2 & 0 & -s & 0\\
0 & 1 & s^{2q} & 0 & 0  & -s^q\\
0 & 0 & 1   & 0 & 0  & 0\\
0 & 0 & 2s^{q+1}  & 1 & -s^q & -s\\
0 & 0 & 2s   & 0 & -1 & 0\\
0 & 0 & 2s^q   & 0 &  0 &-1
\end{pmatrix},\quad
Y=\begin{pmatrix}
0&0&1&0&0&0\\
1&0&0&0&0&0\\
0&1&0&0&0&0\\
0&0&0&0&1&0\\
0&0&0&0&0&1\\
0&0&0&1&0&0
\end{pmatrix} .$$
The coefficients of $\chi_{XY}(t)$ of degrees $1,2,3$ are respectively:
$$
(s^2 -s^q -a^q) - b^q/c; \quad
(-s^{q+2} + s + s^{2q}) - b/c; \quad
s^3 - 4s^{q+1}t + s^{3q} + 2.
$$

Equating the coefficients of the terms $t^ i$ in $\chi_{XY}(t)$ and
$\chi_{xy}(t)$ we get the following conditions $P_i$:
$$\begin{array}{cc}
P_1=c(s^2- s^q- a^q) - b^q=0; &
P_2=c(-s^{q+2} + s + s^{2q}) - b=0;\\
P_3=s^3 - 4s^{q+1} + s^{3q} + 2=0.
\end{array}$$
Here, we are considering $b, c$ in the definition of $y$  as free parameters, with
the only restrictions $b\in \F_{q^2}$, $0\ne c\in \F_q$.
Condition $P_2=0$ gives $
b=c(-s^{q+2} + s + s^{2q})$.
Then $P_1=0$ gives $a=s^{q+2} - 2s$.
Now, recalling the definition of $b$ and $c$ as functions of $a$ and
substituting the previous expressions for $a$ and $b$, we obtain 
$$(s^{q+1}-3)(s^{2q+2}-s^{3q}-2s^{q+1}+2)=0.$$
If $s^{q+1}=3$, then  $a=s$. From $P_3=0$ it follows $a^{q+1}=3$ and
$a^6-10a^3+27=0$.
Setting $a^ q= 3a^ {-1}$, we get $\gamma=a^3 - 6a^{q+1}+ a^{3q} + 8=\frac{a^
6-10a^3+27}{a^ 3}=0$, which implies $H$ reducible.

So, assume $P'=s^{2q+2}-s^{3q}-2s^{q+1}+2=0$.
In this case, the resultant between $P_3$ and $P'$, with respect to $s^q$, is 
$(s^3-1)(s^3-4)^3$. 
If $s^3=1$, from $P_3=0$ we get $s^{q+1}=1$, whence $a=-s$. 
This means $a=-\omega^i$ and
$a^q=-\omega^{2i}$ with $i=1,2$. However, for
these values $H$ is reducible ($\det(J)=0$).
If $s^3=4$, from $P_3=0$ we get $2s^{q+1}=5$.
Taking the third power of both sides, $8(s^3)^q s^3=125$, whence $128=125$, i.e. $p=3$.
This implies $s^3=1$, case that we have already excluded.
 Similarly, we can deal with the case  $\tau=\sigma$.
\end{proof}

\begin{lemma} \label{class S}
Assume $H$ absolutely irreducible. Then $H$ is not contained in any maximal subgroup $M$
in class $\mathcal{S}$.
\end{lemma}

\begin{proof}
We refer to \cite[Table 8.27, page 391]{Ho}.
Case by case, we
exclude that $H$ is
contained in $M$ with projective image $\overline M$.\\
\textbf{Case (a)} $\overline M= \PSU_3(q^2)$, $q$ odd. By Lemma \ref{U3} this cannot occur.

From now on  we may assume $q=p\geq 3$. \\
\textbf{Case (b)} $\overline{M} \in \{ \Alt(6), \Alt(6).2_3, \Alt(7), PSL_3(4), PSL_3(4).2_1\}$.
Simply observe that the only possible values for the order of the projective image of $xy$
are $\leq 8$, which is  excluded by Lemma \ref{powers}.\\
\textbf{Case (c)} $M=Z(\SU(6,q^ 2))\circ \SL_2(11)$, $2<p\equiv 2,6,7,8,10\pmod {11}$.
Notice that the only involution in $\SL_2(11)$ is $-I$.
Thus, if $H\leq M$, it follows $x\in Z(M)$ and so $x$ should commute with $y$, but this does not
happen.\\
\textbf{Case (d)} $M=6_1.\PSU_4(9)$, $p\equiv 5\pmod{12}$. \\
Considering the eigenvalues of an irreducible representation of degree
$6$ of $M$, we have that $x$ belongs to the class $2c$ and $y\in \{3i, 3j\}$ (in GAP
notation). 
We determine the possible orders of $xy$ calculating the
structure constants using the character table of $M$. Notice that $xy$ cannot have
eigenvalues of
multiplicity greater than $1$.

If  $y\in 3i$, then  $xy\in \{9g,9h\}$. However, in this case $(xy)^9$ is scalar. 
If $y\in 3j$, then $xy \in \{7a,7b,8a,8b,9g,9h, 14a, 14b,21a, 21b,21c,21d,
24a,24b,24c,$ $24d, 42a,42b,42c,42d \}$. 
However, for any choice of $xy$ we obtain a contradiction with Lemma \ref{powers},
except when $xy\in \{24a, 24b, 24c, 24d\}$. In these last cases the characteristic
polynomial of $xy$  is  $t^6+\omega^{2j} t^4+\omega^{j} t^2+1$  ($j=1,2$).
Comparison with \eqref{char6} gives $b+ac=b+c\omega^j=0$, whence $a=\omega^j$.
Since $a\not \in \F_q$, $a^q=\omega^{2j}$ and so $b+ac=-2\omega^j$, a contradiction.\\
\textbf{Case (e)} $M=6_1.\PSU_4(9).2_2$, $p\equiv 11 \pmod {12}$.\\
Considering the eigenvalues of an irreducible representation of degree $6$ of $M$, we
have $x\in 2c$ and $y\in \{3i ,  3j\}$. 
We proceed as done in item (1).  If  $y\in 3i$, then  $xy$ must belong to class $9g$, a
contradiction with Lemma \ref{powers}. 
If $y\in 3j$, then $xy \in \{7a,8a,9g,  14a, 21a,
21b, 24a, 24b, 42a, 42b\}$. However, in
each case we obtain a contradiction with Lemma \ref{powers},
except when $xy\in \{24a, 24b\}$. As before, in these last cases the characteristic
polynomial of $xy$  is  $t^6+\omega^{2j} t^4+\omega^{j} t^2+1$  ($j=1,2$), which leads to
a contradiction.
\end{proof}

It can be shown that the group $\SU_6(4)$ cannot be generated by an element $x$ having 
similarity invariants $(t-1), (t-1), (t^2-1), (t^2-1)$
and an element $y$ of order $3$. However, choosing $x$ 
with similarity invariants $(t^2-1), (t^2-1), (t^2-1)$
we get the following:

\begin{proposition}\label{U62}
The group $\SU_6(4)$ is $(2,3)$-generated. 
\end{proposition}

\begin{proof}
Take the following two matrices of $\SL_6(4)$:
\begin{equation} \label{genF4}
x=x_6=\left(\begin{array}{ccccccc}
0 & 1 & 0 & 0 & 0 & 0\\
1 & 0 & 0 & 0 & 0 & 0\\
0 & 0 & 1 & 1 & 0 & \omega\\
0 & 0 & 1 & 0 & \omega^2 & \omega\\
0 & 0 & 0 & \omega & 1 & \omega^2\\
0 & 0 & \omega^2 & \omega^2 & \omega & 0
\end{array}\right),
\quad 
y=y_6=\left(\begin{array}{ccccccc}
0 & 0 & 1 & 0 & 0 & 0\\
1 & 0 & 0 & 0 & 0 & 0\\
0 & 1 & 0 & 0 & 0 & 0\\
0 & 0 & 0 & 0 & 0 & 1\\
0 & 0 & 0 & 1 & 0 & 0\\
0 & 0 & 0 & 0 & 1 & 0
\end{array}\right).
\end{equation}
Then $x^2=y^3=1$. Moreover $x^ Tx^ \sigma=y^ Ty^ \sigma=I$ and so $H=\langle
x,y\rangle\leq \SU_6(4)$. 
Notice that $xy$, $[x,y]$ and $([x,y]^2xy)$ have order $11$, $9$ and
$21$, respectively. The only maximal subgroups $M$ of $\SU_6(4)$ whose orders are
divisible by $7\cdot9\cdot 11$ are of type $3^.M_{22}$.  On the other hand, $3^.M_{22}$
does not have elements of order $9$, whence $H=\SU_6(4)$.
\end{proof}

\begin{theorem}\label{main6}
Let $x,y$ as in \eqref{generatorsU6}, with $a\in \F_{q^ 2}\setminus \F_q$, such that $a^
{q+1}\neq 4 $ and
\begin{itemize}
\item[(i)] $ a^3 - 6a^{q+1}+ a^{3q} + 8 \ne 0$;
\item[(ii)] $a^{2q+2}-5a^{q+1}+8\ne 0$;
\item[(iii)] $\F_p[a^3]=\F_{q^2}$.
\end{itemize}
Then $H=\langle x,y\rangle=\SU_6(q^2)$. Moreover, if $q^2\neq 2^2$, then there exists
$a\in \F_{q^2}^\ast$ satisfying conditions {\rm (i)} to {\rm (iii)}.
\end{theorem}

\begin{proof}
By Proposition \ref{U62} we may assume $q>2$.
Conditions (i) and (ii) imply that $H$ is absolutely irreducible (Theorem
\ref{irreducibility}).
By Lemma \ref{unitary} we have $H\leq SU_6(q^ 2)$.
Let $M$ be a maximal subgroup
of $\SU_6(q^2)$ which contains $H$. 
As observed before, Conditions (i) and (ii) imply that 
$M\not\in \mathcal{C}_1\cup \mathcal{C}_3$.
Moreover $M\not\in \mathcal{C}_2$ by Lemma \ref{monomial}. 
 From Condition (iii) and  Lemma \ref{minimal field} we get 
$M\not\in \mathcal{C}_5$. 
 From Lemma \ref{tensor} it follows  $M\not\in \mathcal{C}_4$.
Finally, the case $M\in \mathcal{S}$ is excluded by Lemmas \ref{U3} and \ref{class S}.

We now prove that, for $q>2$, there exists $a\in \F_{q^ 2}\setminus \F_{q}$, with
$a^{q+1}\neq 4$, satisfying Conditions (i) to (iii).
Take an element $a\in \F_{q^2}^\ast$ of order $q^ 2-1$. Then $a^{q+1}\neq 4$. Indeed
assume $a^{q+1}=4$: clearly $p\neq 2$ and 
so $a^{\frac{q+1}{2}}=\pm 2$, whence the contradiction $a^{\frac{q^ 2-1}{2}}=1$.
By Lemma 2.4. of \cite{PT}, we have $\F_p[a^3]=\F_{q^2}$, since we are assuming $q>2$.
Now, if $q\geq 23$ there are at least $5q+3$  elements of $\F_{q^2}$ of order $q^2-1$ (use
\cite[Lemma 2.1]{PT} 
for $q\geq 127$ and direct computations otherwise). Thus,
 our claim is proved since the elements which do not satisfy (i) or (ii) are at most
$5q+2$. 
For $2<q\leq 19$ take $a$ whose minimal polynomial $m_a(t)$ over $\F_p$ is given 
below.
\begin{center}
\begin{tabular}{cc|cc|cc}
$q$ & $m_a(t)$ & $q$ & $m_a(t)$ &$q$ & $m_a(t)$   \\\hline
$3,13$ & $t^2+t+2$ & $4,9$ & $t^4+t^3-1$  & $5,11$ & $t^2+t-3$  \\
$7,17,19$ & $t^2+t+3$ & $8$ & $t^6+t+1$  & $16$ & $t^8+t^6+t^5+t+1$\\
\end{tabular}
\end{center}
Then $a$ satisfies (i), (ii) and (iii) and hence $H=\SU_6(q^2)$, for all $q> 2$.
\end{proof}

\section{Towards a uniform  $(2,3)$-generation of $\SU_n(q^2)$, $n\geq 6$.}

Write $n=3m+r$ with $0\le r \le 2$. Let $V$ be the vector space over $\F_{q^2}$ with
basis:
\begin{equation}\label{basis}
B=\left\{e_i\mid 1\le i \le 3m\right\},\  {\rm if}\ r=0,\quad \widehat B=\left\{v_i\mid 1\le i \le r\right\}\cup  B,\  {\rm if}\ r>0.
\end{equation}
For $n>6$, we consider the following 
elements of $\SL_n(q^2)$:
\begin{equation}
x_n=\nu_n x_6,\ y_n=\mu_n y_6,
\end{equation} 
where $x_6$ and $y_6$ act on $\left\langle e_{n-j} \mid 0\le j \le 5\right\rangle$ as in 
\eqref{generatorsU6} if $q> 2$, as in \eqref{genF4} if $q=2$. 
The remaining basis vectors are fixed by $x_6,y_6$. 

The matrices $\nu_n$ and $\mu_n$, $n\ge 7$, are monomial and act as follows.
\begin{equation}\label{qne2}
 q>2:\enskip \nu_n=\eta_1\eta_2 \prod_{k=1}^{m-2}\left(e_{3k}, e_{3k+1}\right),
\qquad \mu_n=\prod_{k=0}^{m-3}\left(e_{3k+1},e_{3k+2}, e_{3k+3} \right)
\end{equation}
where $\eta_1 =$ id if $r=0$, $\eta_1=\prod_{i=1}^{r}\left(v_i, e_i\right)$ if $r>0$;
$\eta_2=$ id when $q$ is even  or $n=8$, $\eta_2\left(e_{3m-4}\right)=\pm e_{3m-4}$, so that $\det(x_n)=1$, otherwise.

For $q=2$, $n\ne 7,8$:
$$
\nu_n=\theta \prod_{k=1}^{m-3}\left(e_{3k}, e_{3k+1}\right)\left(e_{3m-6},e_{3m-5},
e_{3m-4}\right), \enskip \mu_n=\delta \prod_{k=0}^{m-3}\left(e_{3k+1},e_{3k+2}, e_{3k+3}
\right),
$$
where $\theta=\left(e_1,e_2\right)$ if $r=0$,\enskip $\theta=\left(e_1,v_1\right)$ if $r=1$, \enskip
$\theta=\left(e_1,v_1\right)\left(e_2,v_2\right)$ if $r=2$;\enskip
$\delta=\textrm{id}$ if $r=0,1$ and $\delta(v_i)=\omega^i v_i$ if $r=2$. 

For $n=7,8$, $\nu_7 = \left(v_1,e_1,e_2\right)$, $\nu_8 =
\left(v_1,e_1,v_2,e_2\right)$. 

It is easy to see that $x_n$ and $y_n$ have respective orders $2,3$ and
belong to  $\SU_n(q^2)$.
Computer evidences, combined with the permutational methods of \cite{T}, \cite{TWG} and \cite{TW},
strongly suggest that $\left\langle x_n,y_n\right\rangle = \SU_n(q^2)$ 
whenever $\left\langle x_6,y_6\right\rangle = \SU_6(q^2)$. The cases $n=8$, $q=3,4$
require a slight modification of the action of $y$ on $\left\langle
v_1,v_2\right\rangle$.

\end{document}